\documentclass[10pt]{amsart}
\usepackage[centertags]{amsmath}
\usepackage{amsfonts}
\usepackage{amssymb}
\usepackage{amsthm}
\usepackage{newlfont}
\usepackage{bm}
\usepackage{amsmath, amsthm, amscd, amsfonts, amssymb, graphicx, color}
\usepackage[bookmarksnumbered, colorlinks, plainpages]{hyperref}
\hypersetup{colorlinks=true,linkcolor=red, anchorcolor=green, citecolor=red, urlcolor=red, filecolor=magenta, pdftoolbar=true}
\newlength{\defbaselineskip}
\newcommand{\setlinespacing}[1]%
{\setlength{\baselineskip}{#1 \defbaselineskip}}

\def\AA{{\mathcal A}}

\def\CC{{\mathcal C}}
\def\DD{{\mathcal D}}

\def\MM{{\mathcal M}}

\def\RR{{\mathcal R}}

\def\TT{{\mathcal T}}

\def\bbC{\mathbb{C}}

\def\1{\mathbf{1}}

\def\bbC{\mathbb{C}}

\def\bfC{\mathbf{C}}
\def\bfD{\mathbf{D}}
\def\bfT{\mathbf{T}}
\def\bfCD{\mathbf{CD}}
\def\bfAB{\mathbf{AB}}
\def\bfS{\mathbf{S}}
\def\bfA{\mathbf{A}}
\def\bfB{\mathbf{B}}
\def\bfI{\mathbf{I}}
\def\bfN{\mathbf{N}}
\newcommand{\de}{\bigtriangleup}

\theoremstyle{plain}
\newtheorem{thm}{Theorem}[section]
\newtheorem{cor}[thm]{Corollary}
\newtheorem{lem}[thm]{Lemma}
\newtheorem{prop}[thm]{Proposition}

\newtheorem{rmk}[thm]{Remark}
\theoremstyle{definition}

\def\dss{\displaystyle}
\theoremstyle{remark}

\numberwithin{equation}{section}

\begin{document}
	
	\def\dss{\displaystyle}
	
	\title{On Some Algebraic Properties of Block Toeplitz Matrices with Commuting Entries}
	
	\author[Muhammad Ahsan Khan and Ameur Yagoub ]{Muhammad Ahsan Khan $^1$, Ameur Yagoub $^2$}
	
	\address{$^{1}$ Department of Mathematics, University of Kotli Azad Jammu $\&$ Kashmir, Kotli 11100, Azad Jammu $\&$ Kashmir, Pakistan.}
\email{ahsankhan388@hotmail.com}	
	\address{$^{2}$  Laboratoire de math\'ematiques pures et appliqu\'ees. Universit\'e de Amar Telidji. Laghouat, 03000. Algeria.}
	\email{a.yagoub@lagh-univ.dz}
	\begin{abstract}
		Toeplitz matrices are ubiquitous and play important roles across many areas of  mathematics. In this paper, we present some algebraic results concerning block Toeplitz matrices with block entries belonging to a commutative algebra $\AA$. The  characterization of normal block Toeplitz matrices with entries from $\AA$ is also obtained. 
	\end{abstract}
	\maketitle
	\section{Introduction}	
	Toeplitz  matrix is  important  due  to  its  typical  property  that  the  entries 
	in  the  matrix  depend  only  on  the  differences  of  the  indices,  and  as  a  result, 
	the  entries  on  its  main  diagonal  as  well  as  those  lying parallel to main diagonal  are  constant. These matrices arise naturally in several ﬁelds of mathematics, as
	well as applied areas as signal processing or time series analysis. The monographs dedicated to the subject are \cite{widom,ioh} and \cite{grenand-szego}.
	
	The corresponding general theory of block Toeplitz matrices is less developed mostly due to intrinsic algebraic difficulties that appear with respect to the scalar case. Block Toeplitz matrices appear in \cite{shalom} but very briefly in \cite{zimmerman,MAKDT,MAK}. 
	
	In \cite{gu-patton}, the authors have proved the variety of  algebraic results concerning scalar Toeplitz matrices.  Among other things, they have obtained the necessary and sufficient condition for the product  $AB-CD$ to be a Toeplitz matrix and $AB-CD=0$, provided that $A,B,C$ and $D$ are Toeplitz matrices. They have also completely characterized normal Toeplitz matrices. We refer the reader to \cite{DI,FKKL, DI1, DIC, DIC1,G, I} where  characterization of normal Toeplitz matrices have been discussed. 
	
	In \cite{MAK1} some  generalization of the results of \cite{gu-patton} concerning the product $AB-CD$ of block Toeplitz matrices has been made. Apart from this \cite{MAK1} has also classified normal block Toeplitz matrices where the entries are taken from the algebra of  scalar diagonal matrices.  We pursue here this investigation, obtaining natural generalization of the results of \cite{gu-patton} by taking entries of the block Toeplitz matrices from a fixed commutative algebra of scalar matrices. We will give new proofs, and refinements of some of the results of \cite{MAK1} and \cite{gu-patton} in a more natural way then that we obtained in \cite{MAK1}. Most importantly, we obtained the criteria for  characterizing  normal block Toeplitz matrices with commuting entries.

	The remaining paper is organized as follows: By means of Section~2, we want to make sure that the reader has become familiar to basic notations and useful facts, needed when we are going to start the main work in upcoming sections. In Section~3 we will provide the generalization and refinements of  the results of \cite{gu-patton} and \cite{MAK1} concerning product of block Toeplitz matrices with commuting entries. Section~4 is concerned with the commutation of certain block Toeplitz matrices. In the last Section, we obtain the characterization of normal block Toeplitz matrices with commuting entries.
	\section{Preliminaries}	
	As usual  $\bbC$ stands for the set of complex numbers.  We designate by $\MM_{n}$ the algebra of $n\times n$ matrices with entries from $\bbC$. We will prefer to label the rows and columns of $n\times n$ matrices from $1$ to $n$; so $A\in\MM_{n}$ is written $A=(a_{i,j})_{i,j=1}^{n}$ with $a_{i,j}\in\bbC$. Then $\TT_{n}\subset\MM_{n}$ is the space of scalar Toeplitz matrices.  
	We will mostly be interested in block matrices, that is, matrices whose elements are
	not necessarily scalars, but elements in $\MM_{d}$. Thus a block Toeplitz matrix is actually an $ nd\times nd $ matrix, but which has been decomposed in $ n^2 $ blocks of dimension $ d $, and  these blocks are constant parallel to the main diagonal.\\
	Throughout in this paper, we will denote $n\times n$ block matrices by bold capital letters.  As a preparation, let us remember that the entries of scalar Toeplitz matrices are complex numbers. Therefore, in order to obtain relevant results for block Toeplitz matrices, as a first step  we will assume that their entries belong to a
	fixed  commutative algebra of $\MM_{d}$, that we will denote by $\AA$. If $\AA$ is a subalgebra of $\MM_{d}$, then $\AA^\prime$ denotes the commutant of $\AA$ (the set of all $d\times d$ matrices commuting with every element of $\AA$). It is straightforward  to show  that $\AA^\prime$ is also an algebra. 
	
	If $A=\begin{pmatrix}
		&0\\
		& A_1\\
		& A_2\\
		
		&\vdots \\
		&A_{n-1}
	\end{pmatrix}$ and  
	$\Omega=\begin{pmatrix}
		&0\\
		&\Omega_1\\
		&\Omega_2\\
		&\vdots \\
		&\Omega_{n-1}
	\end{pmatrix}$ are column vectors with entries from  $\AA$, then let $\bfT(A,\Omega)$ denote the $n\times n$ block Toeplitz matrix of the form: 
	\begin{equation}\label{toeplitzmatrix}
		\mathbf{T}(A, \Omega)=
		\begin{pmatrix} 
			0 &\Omega_{1}^*&  \Omega_{2}^*&\ldots  & \Omega_{n-1}^*\\
			A_{1}&0       &  \Omega_{1}^* &\ldots  & \Omega_{n-2}^*\\
			A_{2}& A_{1}     &  0        &\ldots  & \Omega_{n-3}^*\\
			\vdots    & \vdots      & \vdots        & \ddots\\
			A_{n-1}&  A_{n-2} &  A_{n-3}   & \ldots & 0
		\end{pmatrix}
		.\end{equation}
	Then $\bfT(A,\Omega)+\bfA_0$ will describe the general block Toeplitz matrix . If $\AA$ is a commutative subalgebra of $\MM_{d}$, then we will use the following notations:
	\begin{itemize}
		\item $\MM_{n}\otimes\AA$ is the collection of $ n\times n $ block matrices whose entries all belong to $ \AA $;
		\item $\TT_{n}\otimes\AA$ is the collection of $ n\times n $ block Toeplitz matrices  whose entries all belong to $ \AA $;	
		\item $\DD_{n}\otimes\AA$ is the collection of $ n\times n $ diagonal block Toeplitz matrices whose entries all belong to $ \AA $;
		\item $\CC\otimes\AA$ is the collection of all $n\times 1$ block matrices whose entries all belong to $\AA$ ;
		\item $\RR\otimes\AA$ is the collection of all $1\times n$ block matrices whose entries all belong to $\AA$. 
	\end{itemize}
	It is obvious that  $\DD_{n}\otimes\AA\subset\TT_{n}\otimes\AA\subset\MM_{n}\otimes \AA$. If $A\in\CC\otimes\AA$ (or $\RR\otimes\AA$) and $X\in\AA^\prime$, then we will use the notation $X\diamond A$ to indicate that $X$ is multiplied in a usual way with every entry of $A$. 
	The results of our paper are related to the vectors which define upper and lower triangular parts of block Toeplitz matrices, so our notation for block Toeplitz matrices also pertains to these vectors.

	Throughout whenever we will use the notation $ 	\mathbf{T}(A, \Omega) $, $ A $ and $ \Omega $ will be in $\CC\otimes\AA$ with  the first entry $0$.  Following this notation it is immediate that $\bfT(A,\Omega)^*=\bfT(\Omega,A)$.  
	For fixed $1\leq k\leq n$, let $P_{k-1}$ be the vectors in $\RR\otimes\AA$ whose entry at $k-1$ position is $I$ and all other entries are zero; thus 
	\begin{align*}
		P_0&=
		\begin{pmatrix}
			I ,0,0,\cdots ,0
		\end{pmatrix}\\
		P_1&=
		\begin{pmatrix}
			0 ,I,0,\cdots ,0
		\end{pmatrix}\\
		P_2&=
		\begin{pmatrix}
			0 ,0,I,\cdots ,0
		\end{pmatrix}
	\end{align*}
	etc. 
	Let $I\in\AA$ be the identity matrix and $\bfS$ denote the matrix consisting of $I$ along the subdiagonal and zero elsewhere, i.e., 
	\[\bfS=
	\begin{pmatrix}
		0 &  0  &0   &\ldots  &  0     &     0\\
		I    &  0 &0 &\ldots  &  0     &     0\\
		0  & I  & 0 & \ldots  &  0     &     0\\
		\vdots& \vdots & \vdots &\ddots& \vdots\\
		0 &  0  & 0&\ldots & I&0
	\end{pmatrix}. 
	\]
	Note that, $ \bfS^n=\bfS^*{}^n =0$. If $X\in\AA^\prime$, then we denote the matrix $\bfS+X\diamond P_0^*P_{n-1}$ by $\bfS_X$. For 
	$A=
	\begin{pmatrix}
		&0\\
		&A_{1}\\
		&\vdots\\
		&A_{n-1}
	\end{pmatrix}\in\CC\otimes\AA
	$, we define
	$\widetilde{A}=
	\begin{pmatrix}
		0\\
		{A_{n-1}^*}\\
		\vdots \\
		{A_{1}^*}
	\end{pmatrix}.
	$
	If $\bfA\in\MM_{n}\otimes\AA$, then the displacement matrix for $\bfA$ is defined as
	\[
	\de(\bfA)=\bfA-\bfS\bfA\bfS^*
	.\]  See \cite{HK} and \cite{KS}  for other types of displacement matrices. Note, in particular that if $\bfI\in\MM_{n}\otimes\AA$, then $\de(\bfI)=\bfI-\bfS\bfS^*=P_0^*P_0$.

	The below written results from \cite{gu-patton} are also valid for block  matrices with entries from $\AA$. We will add their proofs just for completeness.
	\begin{lem}
		If $\bfA\in\MM_{n}\otimes\AA$, then  $\bfA=\dss\sum_{k=0}^{n-1}\bfS^k\de(\bfA){\bfS^*}^{k}$.
	\end{lem}
	\begin{proof}
		\begin{align*}
			\dss\sum_{k=0}^{n-1}\bfS^k(\bigtriangleup(\bfA))\bfS^{k*}
			&=\dss\sum_{k=0}^{n-1}\bfS^k(\bfA-\bfS \bfA\bfS^*)\bfS^{k*}\\
			&=\dss\sum_{k=0}^{n-1}(\bfS^k\bfA\bfS^{k*}-\bfS^{k+1}\bfA\bfS^{k+1*})=\bfA-\bfS^n\bfA\bfS^{n*}=\bfA.
		\end{align*}
	\end{proof}
	Thus to show that $\bfA=0$, it will be sufficient to show that $\de(\bfA)=0$. 
	We have the following analogue of the Lemma 2.2 of \cite{gu-patton} for block Toeplitz matrices with commuting entries.
	\begin{lem}\label{Toeplitz}
		$\bfA\in\MM_{n}\otimes\AA$  is Toeplitz if and only if there exist vectors  $A,\Omega\in\CC\otimes\AA$  such that $\bigtriangleup(\bfA)=AP_{0}+P_{0}^{\ast}\Omega^*.$
	\end{lem}
	\begin{proof}
		Suppose that $\bfA=\bfT(A,\Omega)+\bfA_0\in\TT_{n}\otimes\AA$. Since the displacement matrix for $\bfA$ is defined as $\de(\bfA)=\bfA-\bfS\bfA\bfS^{\ast}.$
		Then  simple computation yields that 
		\[\de (\bfA)=
		\begin{pmatrix} 
			A_{0} & \Omega_{1}^* & \Omega_{2}^*&\ldots  & \Omega_{n-1}^*\\
			A_{1}&  0    &  0   &\ldots  & 0\\
			A_{2}&  0    &  0   &\ldots  & 0\\
			\vdots& \vdots   & \vdots & \ddots\\
			A_{n-1}&  0 &  0  & \ldots & 0
		\end{pmatrix}
		.\] 
		If we take 
		$
		A = \begin{pmatrix}
			A_{0}\\
			A_{-1}\\
			\vdots\\
			A_{1-n}
		\end{pmatrix}
		$ and $\Omega =
		\begin{pmatrix}
			0 \\
			\Omega_{1}\\
			\vdots\\
			\Omega_{n-1}
		\end{pmatrix},
		$ then one can easily verify that $\de(\bfA)=AP_{0}+P_{0}^{\ast}\Omega^*.$
		For the converse, let $\bfA=(A_{ij})_{i,j=1}^{n}\in\MM_{n}\otimes\AA$. 
		Suppose then  that $
		A = \begin{pmatrix}
			A_{0}\\
			A_{1}\\
			\vdots\\
			A_{n-1}
		\end{pmatrix}
		$
		and $\Omega=
		\begin{pmatrix}
			\Omega_{0}\\
			\Omega_{1}\\
			\vdots\\
			\Omega_{n-1}
		\end{pmatrix}
		$ be vectors in $\CC\otimes\AA$, since we have 
		\[
		\de(\bfA)=AP_{0}+P_{0}^{\ast}\Omega^*
		\] $\implies$  
		\[
		\bfA=\bfS\bfA\bfS^{\ast}+AP_{0}+P_{0}^{\ast}\Omega^*
		\]$ \implies$ 
		\[\bfA=
		\begin{pmatrix} 
			A_{0}+\Omega_{0}^{\ast}&\Omega_{1}^{\ast}&\Omega_{2}^{\ast}&\ldots&\Omega_{n-1}^{\ast}\\
			A_{1} &  A_{1,1} &  A_{1,2} &\ldots & A_{1,n-1}\\
			\vdots & \vdots  & \vdots  & \ddots\\
			A_{n-1}&   A_{n-1,1} &  A_{n-1,2} & \ldots & A_{n-1,n-1}
		\end{pmatrix}.
		\]  
		Compairing entries along the diagonals yields that 
		$A_{i_{1},j_{1}}=A_{i_{2},j_{2}}$, whenever $i_{1}-j_{1}=i_{2}-j_{2}$, where $0\leq i_{1}, i_{2},j_{1},j_{2}\leq n-1$, i.e., $\bfA\in\TT_{n}\otimes\AA$ .
	\end{proof}
	\section{ Product of Block Toeplitz Matrices with Commuting Entries}
	In this section we will obtain the results  concerning the product $\bfAB-\bfCD$ ,where $\bfA,\bfB,\bfC$ and $\bfD$ are block Toeplitz matrices with entries from $\AA$.\\
	We start  with the following Lemma, which describes the structure of the displacement matrix for the product of two block Toeplitz matrices with commuting entries. 
	\begin{lem}
		Let $C=\begin{pmatrix}
			&0\\
			& C_1\\
			& C_2\\
			&\vdots \\
			&C_{n-1}
		\end{pmatrix}$, $\Gamma=\begin{pmatrix}
			&0\\
			&\Gamma_1\\
			&\Gamma_2\\
			&\vdots \\
			&\Gamma_{n-1}
		\end{pmatrix}$, $D=\begin{pmatrix}
			&0\\
			&D_1\\
			&D_2\\
			&\vdots\\
			&D_{n-1}
		\end{pmatrix}$, and 
		$\Theta =\begin{pmatrix}
			&0\\
			&\Theta_1\\
			&\Theta_2\\
			&\vdots \\
			&\Theta_{n-1}
		\end{pmatrix}$, be vectors in $\CC\otimes\AA$.  If $\mathbf{C}=\mathbf{T}(C,\Gamma)+\mathbf{C}_0$ and $\mathbf{D}=\mathbf{T}(D,\Theta)+\mathbf{D}_0$, then
		
		\begin{equation}\label{Toeplitzproduct}
			\de(\mathbf{CD})=C\Theta^*-\widetilde{\Gamma}\widetilde{D}^*+\left[\bfC D+\bfD_0C+\bfC_0\bfD_0P_0^*\right]P_0+P_0^*\left[\Gamma^*\bfS\bfD\bfS^*+\Theta^*\bfC_0\right].
		\end{equation}.
	\end{lem}
	\begin{proof}
		Let $\widehat{\bfC}=\bfT(C,\Gamma)$ and $\widehat{\bfD}=\bfT(D,\Theta)$. Then we have 
		\begin{align*}
			\de(\bfCD)
			&=\de[(\widehat{\bfC}+\bfC_0)][(\widehat{\bfD}+\bfD_0)]\\
			&=\de[\widehat{\bfC}\widehat{\bfD}+\bfC_0\widehat{\bfD}+\bfD_0\widehat{\bfC}+\bfC_0\bfD_0]\\
			&=\de(\widehat{\bfC}\widehat{\bfD})+\de(\bfC_0\widehat{\bfD})+\de(\bfD_0\widehat{\bfC})+\de(\bfC_0\bfD_0).
		\end{align*}
		Since $\bfC_0,\bfD_0\in\DD_{n}\otimes\AA$, then $\bfS$ commute with $\bfC_0$ and $\bfD_0$ respectively,  therefore last equation above can be written as 
		\begin{equation}\label{eq2}
			\de(\bfCD)=\de(\widehat{\bfC}\widehat{\bfD})+\bfC_0\de(\widehat{\bfD})+\bfD_0\de(\widehat{\bfC})+\bfC_0\bfD_0\de(\bfI).
		\end{equation} 
		By Lemma \ref{Toeplitz}, there exist vectors  $D,\Theta\in\CC\otimes\AA$ such that, $\de(\widehat{\bfD})=DP_0+P_0^*\Theta^*$. Similarly $\de(\widehat{\bfC})=CP_0+P_0^*\Gamma^*$, with $C,\Gamma\in\CC\otimes\AA$. Also we have $\de(\bfI)=P_0^*P_0$. Then \eqref{eq2} becomes 
		\begin{align*}
			\de({\bfCD})&=\de(\widehat{\bfC}\widehat{\bfD})+\bfC_0[DP_0+P_0^*\Theta^*]+\bfD_0[CP_0+P_0^*\Gamma^*]+\bfC_0\bfD_0P_0^*P_0\\
			&=\de(\widehat{\bfC}\widehat{\bfD})+[\bfC_0D+\bfD_0C+\bfC_0\bfD_0P_0^*]P_0+P_0^*[\Theta^*\bfC_0+\Gamma^*\bfD_0]\label{eq:eq0}\tag{*}.
		\end{align*}
		Now by using the definition of $\de$ 
		\begin{align*}
			\de(\widehat{\bfC}\widehat{\bfD})
			&=\widehat{\bfC}\widehat{\bfD}-\bfS\widehat{\bfC}\widehat{\bfD}\bfS^*\\
			&=\widehat{\bfC}\widehat{\bfD}-\widehat{\bfC}\bfS\widehat{\bfD}\bfS^*+\widehat{\bfC}\bfS\widehat{\bfD}\bfS^*-\bfS\widehat{\bfC}[\bfS^*\bfS+P_{n-1}^*P_{n-1}]\widehat{\bfD}\bfS^*\\
			&=\widehat{\bfC}\de\widehat{\bfD}+\de\widehat{\bfC}(\bfS\widehat{\bfD}\bfS^*)-\bfS\widehat{\bfC}P_{n-1}^*P_{n-1}\widehat{\bfD}\bfS^*\\
			&=\widehat{\bfC}[DP_0+P_0^*\Theta^*]+[CP_0+P_0^*\Gamma^*]\bfS\widehat{\bfD}\bfS^*-\widetilde{\Gamma}\widetilde{D}^*.\label{eq:eq3}\tag{**}
		\end{align*}
		Since $P_0\bfS=0$, so the term $CP_0\bfS\widehat{\bfD}\bfS^*=0$. Also $\widehat{\bfC}P_0^*\Theta^*=C\Theta^*$, then (\ref{eq:eq3}) can be written as  
		\begin{align}\label{eq4}
			\de(\widehat{\bfC}\widehat{\bfD})=\widehat{\bfC}DP_0+P_0^*\Gamma^*\bfS\widehat{\bfD}\bfS^*+C\Theta^*-\widetilde{\Gamma}\widetilde{D}^*.
		\end{align}
		Combining \eqref{eq:eq0} and \eqref{eq4} we obtained 
		\begin{equation}\label{eq5}
			\de(\bfCD)=C\Theta^*-\widetilde{\Gamma}\tilde{D}^*+[\widehat{\bfC}D+\bfC_0D+\bfD_0C+\bfC_0\bfD_0P_0^*]P_0+P_0^*[\Gamma^*\bfS\widehat{\bfD}\bfS^*+\Theta^*\bfC_0+\Gamma^*\bfD_0].
		\end{equation}
		Note that $\widehat{\bfC}D+\bfC_0D=\bfC D$ and $\Gamma^*\bfS\widehat{\bfD}\bfS^*+\Gamma^*\bfD_0=\Gamma^*(\bfS\widehat{\bfD}\bfS^*+\bfD_0)=\Gamma^*\bfS\bfD\bfS^*$.
		Therefore \eqref{eq5} becomes $$\de(\mathbf{CD})=C\Theta^*-\widetilde{\Gamma}\tilde{D}^*+[\bfC D+\bfD_0C+\bfC_0\bfD_0P_0^*]P_0+P_0^*[\Gamma^*\bfS\bfD\bfS^*+\Theta^*\bfC_0].$$
	\end{proof}	
	The following result is the most important result of this section. 
	\begin{thm}\label{main}
		Let $A,\Omega, B, \Lambda, C, \Gamma, D $ and $\Theta$ be vectors in $\CC\otimes\AA$ with $0$ in the zeroth component . Let $\bfA=\bfT(A,\Omega)+\bfA_0$, $\bfB=\bfT(B,\Lambda)+\bfB_0$, $\bfC=\bfT(C,\Gamma)+\bfC_0$,  and $\bfD=\bfT(D,\Theta)+\bfD_0$, then
		\begin{itemize}
			\item [(i)] $\bfAB-\bfCD$ or, equivalently  $\bfT(A,\Omega)\bfT(B,\Lambda)-\bfT(C,\Gamma)\bfT(D,\Theta)\in\TT_n\otimes\AA$ if and only if $$A\Lambda^*-\widetilde\Omega\widetilde{B}^*=C\Theta^*-\widetilde{\Gamma}\widetilde{D}^*.$$		
			\item [(ii)]If $\bfAB-\bfCD\in\TT_n\otimes\AA$, then $\bfAB-\bfCD=0$ if and only if  \begin{equation}\label{zerothrelation}
				\bfA B+\bfB_0A+\bfA_0\bfB_0 P_0^*=\bfC D+\bfD_0C+\bfC_0\bfD_0 P_0^*,
			\end{equation} and 
			\begin{equation}\label{adjointrelation}
				\bfB^*\Omega+\bfA_0^*\Lambda+\bfA_0^*\bfB_0^*P_0^*=\bfD^*\Gamma+\bfC_0^*\Theta+\bfC_0^*\bfD_0P_0^*.
			\end{equation}
		\end{itemize}
	\end{thm}
	\begin{proof}
		By Lemma \ref{Toeplitz},
		\begin{align*}
			\de(\bfAB)-\de(\bfCD)&=A\Lambda^*-\widetilde\Omega\widetilde{B}^*-C\Theta^*+\widetilde{\Gamma}\widetilde{D}^*
			+\left[\bfA B+\bfB_0 A+\bfA_0\bfB_0P_0^* -\bfC D -\bfD_0 C-\bfC_0\bfD_0 P_0^*\right]P_0\\
			&+P_0^*\left[\Omega^*\bfS\bfB\bfS^*+\Lambda^*\bfA_0-\Gamma^*\bfS\bfD\bfS^*-\Theta^*\bfC_0\right],
		\end{align*}
		as a consequence, the first four terms are block matrices with 0 on the first row and the first column. On the other hand, the fifth term has nonzero entries only on the first column, and the sixth only on the first row. By Lemma \ref{Toeplitz}, $\bfAB-\bfCD\in\TT_n\otimes\AA$ if and only if  
		$A\Lambda^*-\widetilde\Omega\widetilde{B}^*-C\Theta^*+\widetilde{\Gamma}\widetilde{D}^*=0$, which is the required relation. \\
		(ii) If $\bfAB-\bfCD\in\TT_n\otimes\AA$, then $\bfAB=\bfCD$ if and only if $\de(\bfAB-\bfCD)=0$.  The latter equation holds, i.e., $\bfAB=\bfCD$ if and only if
		\[
		\bfA B+\bfB_0 A+\bfA_0\bfB_0P_0^* =\bfC D +\bfD_0 C+\bfC_0\bfD_0P_0^* 
		\]
		\[
		\Omega^*\bfS\bfB\bfS^*+\Lambda^*\bfA_0=\Gamma^*\bfS\bfD\bfS^*+\Theta^*\bfC_0
		\]
		or 
		\begin{equation}\label{adj}
			\bfS\bfB^*\bfS^*\Omega+\bfA_0^* \Lambda=\bfS\bfD^*\bfS^*\Gamma+\bfC_0^*\Theta.
		\end{equation}
		However \eqref{zerothrelation} and \eqref{adjointrelation} are equivalent to above equations . Since the difference between \eqref{adjointrelation} and \eqref{adj}  \[
		\de(\bfB^*)\Omega+\bfA_0^*\bfB_0^*P_0^*=\de(\bfD^*)\Gamma+\bfC_0^*\bfD_0^*P_0^*, 
		\]
		which is upto  the  adjoint is the zeroth component relation of \eqref{zerothrelation}.
	\end{proof}
	The below written results gather some consequences of Theorem \ref{main}.
	\begin{cor}\label{Alg}
		If $\bfA=\bfT(A,\Omega)+\bfA_0,\bfB=\bfT(B,\Lambda)+\bfB_0$, then $\bfAB\in\TT_{n}\otimes\AA$  if and only if $A\Lambda^*=\widetilde\Omega\widetilde{B}^*$. 
	\end{cor} 
	\begin{cor}
		Let $\bfA=\bfT(A,\Omega)+\bfA_0,\bfB=\bfT(B,\Lambda)+\bfB_0$, then $\bfAB\in\TT_{n}\otimes\AA$  if and only if $\bfB\bfA\in\TT_{n}\otimes\AA$.	
	\end{cor}
	\begin{proof}
		By Theorem \ref{main}, $\bfAB\in\TT_{n}\otimes\AA$ if and only if $A\Lambda^*=\widetilde{\Omega}\widetilde{B}^*
		$ if and only if $B\Omega^*=(B_i\Omega_j)_{i,j}=((A_{n-j-1}\Lambda_{n-i-1})^*)_{i,j}=(\Lambda_{n-i-1}^*A_{n-j-1}^*)_{i,j}=\widetilde{\Lambda}\widetilde{A}^*$ if and only if $\bfB\bfA\in\TT_{n}\otimes\AA$.
	\end{proof}
	The following results is already proved in {\cite{MAKDT}}.
	\begin{thm}\label{conj} 
		Let $\bfA=\bfT(A,\Omega)+\bfA_0$, $\bfB=\bfT(B,\Lambda)+\bfB_0$,  if $\bfAB\in\TT_{n}\otimes\AA$, then $\bfAB=\bfB\bfA$. 
	\end{thm}

	\section{Commutants of $\bfS$, $\bfS^*$, $\bfS_X$ and $\bfS_X^*$ .}
	We start this section with the following proposition.
	
	\begin{prop} $\bfA \in \MM_n\otimes \AA$ is Toeplitz if and only if there exist $A$, $B\in\CC\otimes \AA$ such that $\bfA-\bfS_X \bfA \bfS^*_X= AP_0 + P_0^*B^*.$
	\end{prop}
	\begin{proof}By Lemma 2.2, $\bfA \in \MM_n\otimes \AA$  is Toeplitz if and only if there exist vectors $A'$, $\Omega'\in\CC\otimes \AA$ such that $$\bfA-\bfS\bfA\bfS^*= A'P_0 + P_0^*(\Omega')^*.$$
		if and only if $$\bfA-(\bfS_X-X\diamond P_0^*P_{n-1})\bfA(\bfS^*_X-P_{n-1}^*P_0\diamond X^*)= A'P_0 + P_0^*(\Omega')^*,$$
		if and only if $$\bfA-\bfS_X\bfA \bfS^*_X=-\bfS_X \bfA P_{n-1}^*P_0\diamond X^*- X\diamond P_0^*P_{n-1}\bfA \bfS^*_X+ X\diamond P_0^*P_{n-1}\bfA P_{n-1}^*P_0\diamond X^*+A'P_0 + P_0^*(\Omega')^*,$$
		if and only if \begin{align*}
			\bfA-\bfS_X\bfA\bfS^*_X&=\left[ A'-\bfS_X \bfA X^* \diamond P_{n-1}^*\right] P_0 +P_0^*[ (\Omega')^*- P_{n-1}\diamond X\bfA \bfS^*_X
			+X\diamond P_{n-1}\bfA P_{n-1}^*P_\diamond X^*],\\
			&=A P_0+P_0^*B^*,\end{align*}
		where $A=A'-\bfS_X \bfA X^* \diamond P_{n-1}^*$ and $B=\Omega'- \bfS_X\bfA^*X^*\diamond P_{n-1}^*+X\diamond P_{0}^* P_{n-1}\bfA^*P_{n-1}^*\diamond X^*.$
		
	\end{proof}
	\begin{rmk} Since $P_{n-1}\bfS^*=0$, then 
		\begin{eqnarray*}
			I-\bfS_X\bfS^*_X&=& \bfI-(\bfS+X\diamond P_0^*P_{n-1})(\bfS^*+P_{n-1}^*P_0 \diamond X^*)\\
			&=& \bfI-\bfS\bfS^*-\bfS P_{n-1}^*P_0\diamond X^*-X\diamond P_0^*P_{n-1}\bfS^*-X\diamond P_0^*P_{n-1}P_{n-1}^*P_0\diamond X^*\\
			&=& P_0^*P_{0}-\bfS X^*\diamond P_{n-1}^*P_0 -P_0^*P_{n-1}\diamond X\bfS^*-X\diamond P_0^*P_{n-1}P_{n-1}^*P_0\diamond X^*\\
			&=& P_0^*P_{0} -P_0^*P_{n-1}\diamond X\bfS^*-\bfS_XP_{n-1}^*P_0\diamond X^*\\
			&=& P_0^*P_{0} -\bfS_XP_{n-1}^*P_0\diamond X^*.
		\end{eqnarray*}
		
	\end{rmk}
	\begin{rmk} 
		$\bfS,\bfS_X\in \TT_n\otimes \AA$, with $\bfS = \bfT(P_1,0)$ and $\bfS_X = \bfT(P_1,X^*\diamond\widetilde{P_1})$.
	\end{rmk}
	The following result characterized lower (upper) triangular block Toeplitz matrices among all $n\times n$ block matrices with entries from $\AA$.
	\begin{thm}\label{commutantS}
		If $\textbf{A} \in \mathcal{M}_n\otimes \mathcal{A}$, then the following hold:
		\begin{enumerate} 
			\item[(i)]   $\bfA\bfS=\bfS\bfA$  if and only if $\bfA = \bfT(A,0)+\bfA_0$.
			\item[(ii)] $\bfA\bfS^*=\bfS^*\bfA$ if and only if $ \bfA = \bfT(0,\Omega)+\bfA_0$.
		\end{enumerate}
	\end{thm}
	\begin{proof} 
		We will give the proof only for (i) and left (ii) as an easy exercise for the reader. 
		
		(i) Suppose that  $\bfA\bfS=\bfS\bfA$, then $\de(\bfA)=\bfA\de(\bfI)=\bfA P_0^*P_{0}$. By Lemma \ref{Toeplitz},  $\bfA \in \TT_n\otimes \AA$, with $\bfA = \bfT(A,0)+\bfA_0$. For the converse, let   $\bfA=\bfT(A,0)+\bfA_0$, also $\bfS \in \TT_n\otimes \AA$, then by Corollary \ref{Alg} and Theorem \ref{conj}, it is immediate that $\bfA \bfS=\bfS\bfA$.
	\end{proof}
	
	\begin{thm}\label{commutantSX}
		If $\textbf{A} \in \mathcal{M}_n\otimes \mathcal{A}$, then 
		\begin{enumerate} 
			\item[(i)]$\bfA\bfS_X=\bfS_X\bfA$ if and only if $\bfA = \bfT(A, X^*\diamond \widetilde{A})+\bfA_0$.
			\item[(ii)] $\bfA\bfS_X^*=\bfS_X^*\bfA$ if and only if $\bfA = \bfT(X^*\diamond \widetilde{A},A )+\bfA_0^*$.
		\end{enumerate}
	\end{thm}
	\begin{proof}
		(i)	Suppose that $\bfA\bfS_X=\bfS_X\bfA$, then we have 
		\begin{align*}
			\de(\bfA)&=\bfA-\bfS\bfA\bfS^*\\
			&= \bfA-[\bfS_X-X\diamond P_0^*P_{n-1}]\bfA[\bfS^*_X-P_{n-1}^*P_0\diamond X^*]\\
			&= \bfA-\bfS_X\bfA\bfS^*_X+\left[ \bfS_X \bfA X^*\diamond P_{n-1}^*\right] P_0 +P_0^*[  P_{n-1}\diamond X\bfA \bfS^*_X-X\diamond P_{n-1}\bfA P_{n-1}^*P_0\diamond X^*]\\
			&= \bfA(\bfI-\bfS_X\bfS^*_X)+\bfA\left[ \bfS_X X^*\diamond P_{n-1}^*\right] P_0 +P_0^*\left[ X \diamond P_{n-1}\bfA \bfS^*\right]\\
			&= \bfA P_0^*P_{0}+P_0^*\left[ X \diamond P_{n-1}\bfA \bfS^*\right].
		\end{align*}
		By Lemma \ref{Toeplitz}, $\bfA \in \TT_n\otimes \AA$, since $X\diamond P_{n-1}\bfA \bfS^*= X \diamond\widetilde{A}^*$, then  $\bfA = \bfT(A, X^*\diamond \widetilde{A})+\bfA_0$. For the converse, since $\bfA, \bfS_X \in \TT_n\otimes\AA$, then Corollary \ref{Alg} and Theorem \ref{conj} imply that , $\bfA \bfS_X=\bfS_X\bfA$. The proof of (ii) is similar to the proof of (i).
	\end{proof}
	\begin{cor}\label{5}
		If $\bfA=\bfT(A,\Omega)+\bfA_0$, $\bfB=\bfT(B,\Lambda)+\bfB_0 $, such that $\bfA$ and $\bfB$ commutes with $\bfS_X$  for some $X\in\AA^\prime$, then $\bfAB \in \TT_n\otimes \AA$. 
	\end{cor}
	\begin{proof}
		Suppose that $\bfA=\bfT(A,\Omega)+\bfA_0$, $\bfB=\bfT(B,\Lambda)+\bfB_0 $, such that for some $X\in\AA^\prime$, $\bfA$ and $\bfB$ commutes with $\bfS_X$ then it follows from Theorem \ref{commutantSX}, that  $\Omega=X^*\diamond\widetilde{A}$ and $\Lambda=X^*\diamond\widetilde{B}$, we have then   
		\[
		A\Lambda^*=A(X^*\diamond\widetilde{B})^*=X\diamond A\widetilde{B}^*=\widetilde{\Omega}\widetilde{B}^*.
		\]
		Therefore by Corollary \ref{Alg}, $\bfAB\in\TT_{n}\otimes\AA$. 
	\end{proof}
	
	\section{Characterization of Normal Block Toeplitz Matrices with Commuting Entries}
	In this section we will characterize normal block Toeplitz matrices with entries from a commutative algebra $\AA$. 
	We start with the following Lemma.
	\begin{lem}\label{normal}
		If $\bfA=(A_{i,j})_{i,j=1}^{n}\in\MM_{n}\otimes\AA$, then $\bfA$ is normal if and only if 
		\[
		\sum_{k=1,k\neq p}^{n}[A_{k,p}^*A_{k,p}-A_{p,k}A_{p,k}^*]=0\quad\hbox{for every }p = 1,2 ,\cdots, n,\]
		and 
		\[\sum_{k=1}^{n}[A_{k,i}^*A_{k,j}-A_{i,k}A_{j,k}^*]=0 \quad\hbox{for every } 1\leq i<j\leq n.
		\]
		
	\end{lem}
	
	\begin{thm}\label{normal th100}
		If $\bfA = \mathbf{T}(A,\Omega)+\bfA_0\in \TT_n\otimes \AA$, 
		then $\bfA$ is normal if and only if for every $s$ and $k$, with $1\leq s,k\leq n-1,$
		\begin{equation*}
			A_sA_k^* +A_{n-s}^*A_{n-k}=\Omega_{s}\Omega^*_{k}+\Omega^*_{n-s}\Omega_{n-k}.
		\end{equation*}
	\end{thm}
	\begin{proof}
		
		Let $\bfN=(N_{i,j})_{i,j=1}^n=\bfA^*\bfA-\bfA\bfA^*$. Since 
		$\bfA\in\TT_{n}\otimes\AA$, it follows from Lemma \ref{normal} that $\bfA$ is normal if and only if
		$$N_{p,p}=\sum_{k=1}^{p-1}\left(\Omega_{p-k}\Omega_{p-k}^*-A_{p-k}^*A_{p-k}\right)+\sum_{k=p+1}^{n}\left(A_{k-p}^*A_{k-p}-\Omega_{k-p}\Omega^*_{k-p}\right)=0$$ 
		and 
		$$N_{i,j}=\sum_{k=1}^{i-1}\left(\Omega_{i-k}\Omega_{j-k}^*-A_{i-k}A_{j-k}^*\right) +\sum_{k=i+1}^{j-1}\left(A_{k-i}^*\Omega_{j-k}^*-\Omega_{k-i}^*A_{j-k}^*\right) +\sum_{k=j+1}^{n}\left(A_{k-i}^*A_{k-j}-\Omega_{k-i}^*\Omega_{k-j}\right)=0,$$
		if and only if 
		\begin{equation}\label{normal T}
			N_{p,p}=\sum_{k=1}^{p-1}\left(\Omega_{k}\Omega_{k}^*-A_{k}^*A_{k}\right)-\sum_{k=1}^{n-p}\left(\Omega_{k}\Omega^*_{k}-A_{k}^*A_{k}\right)=0
		\end{equation}
		and 
		\begin{equation}\label{normal T2} N_{i,j}=\sum_{k=1}^{i-1}\left(\Omega_{k}\Omega_{j-i+k}^*-A_{k}A_{j-i+k}^*\right) +\sum_{k=1}^{j-i-1}\left(A_{k}^*\Omega_{j-i-k}^*-\Omega_{k}^*A_{j-i-k}^*\right)	
			+ 	\sum_{k=1}^{n-j}\left(A_{j-i+k}^*A_{k}-\Omega_{j-i+k}^*\Omega_{k}\right)=0, 
		\end{equation}
		for every  $p = 1,2 ,\cdots, n $ and $1 \leq i < j \leq n ,$ respectively.
		We first calculate the equation (\ref{normal T}),  when $n = 2m$ for some ﬁxed positive integer $m$. If we calculate the diagonal entries of $\bfN$, then
		\begin{equation}\label{normal Td}	
			N_{p,p} =\sum_{k=1}^{p-1}B_k-\sum_{k=1}^{2m-p}B_k= -\left[ \sum_{k=1}^{(2m-p+1)-1}B_k-\sum_{k=1}^{2m-(2m-p+1)}B_k \right] =-N_{2m-p+1,2m-p+1}.
		\end{equation}
		for every $p$, where  $B_k=\Omega_{k}\Omega^*_{k}-A_{k}^*A_{k}.$ Thus, we know that $N_{p,p}=-N_{2m-p+1,2m-p+1}$ for every $p$. So it suffices to consider the diagonal entries $(p,p )$ of $\bfN$ for $p=1,2,\cdots,m$. A simple computation shows from
		(\ref{normal Td}) that
		\begin{align*}N_{m,m}&=\sum_{k=1}^{m-1}B_k-\sum_{k=1}^{m}B_k=-B_m=\left(\Omega_{m}\Omega^*_{m}-A_{m}^*A_{m}\right)=0.
		\end{align*}
		
		By recurrence for every $p=1,2,\cdots, m$, we have 
		\begin{align*}
			N_{p,p} &=\sum_{k=1}^{p-1}B_k-\sum_{k=1}^{2m-p}B_k= \sum_{k=1}^{(p+1)-1}B_k-\sum_{k=1}^{2m-(p+1)}B_k-(B_{p}+B_{2m-p})\\
			&= N_{p+1,p+1}-( B_{p}+B_{2m-p}).
		\end{align*}
		Which implies that  $B_{p}+B_{2m-p}=0$, for every $p=1,2,\cdots,m$. Therefore 
		\begin{equation}\label{normal Td1}
			A_{p}^*A_{p}+A_{2m-p}^*A_{2m-p}=\Omega_{p}\Omega_{p}^*+\Omega_{2m-p}\Omega_{2m-p}^*.	
		\end{equation}
		Next we consider the case $i<j$. we write (\ref{normal T2}) as 
		\[N_{i,j} = \sum_{k=1}^{i-1}C_{r,k}+\sum_{k=1}^{r-1}D_{r,k}+\sum_{k=1}^{2m-j}E_{r,k},\]
		where for every $1\leq r = j-i\leq n-1,$	
		\begin{eqnarray*}
			C_{r,k}&= \Omega_{k}\Omega_{r+k}^*-A_{k}A_{r+k}^*.	\\
			D_{r,k}&=  A_{k}^*\Omega_{r-k}^*-\Omega_{k}^*A_{r-k}^*.\\
			E_{r,k}&=  A_{r+k}^*A_{k}-\Omega_{r+k}^*\Omega_{k}.
		\end{eqnarray*}
		Then 
		\begin{eqnarray*}
			N_{i,i+1}& =& \sum_{k=1}^{i-1}C_{1,k}+0+\sum_{k=1}^{2m-i-1}E_{1,k}	\\
			&= &  \sum_{k=1}^{i}C_{1,k}+\sum_{k=1}^{2m-i-2}E_{1,k}-	C_{1,i}+E_{1,2m-i-1}\\
			&= &   N_{i+1,i+2}-	C_{1,i}+E_{1,2m-i-1}.
		\end{eqnarray*}
		Then $	C_{1,i}-E_{1,2m-i-1}=0$, therefore
		\begin{equation}\label{1} \Omega_{i}\Omega_{i+1}^*+\Omega_{2m-i}^*\Omega_{2m-(i+1)}=A_{i}A_{1+i}^*+A_{2m-i}^*A_{2m-(i+1)},
		\end{equation}
		and 
		\begin{eqnarray*}
			N_{i,i+2}& =& \sum_{k=1}^{i-1}C_{2,k}+D_{2,1}+\sum_{k=1}^{2m-i-2}E_{2,k}	\\
			&= &  \sum_{k=1}^{i}C_{2,k}+D_{2,1}+\sum_{k=1}^{2m-i-3}E_{2,k}-	C_{2,i}+E_{2,2m-i-2}\\
			&= &   N_{i+1,i+3}-	C_{2,i}+E_{2,2m-i-2}.
		\end{eqnarray*}
		Then $	C_{2,i}=E_{2,2m-i-2}=0$,  therefore
		\begin{equation}\label{2} \Omega_{i}\Omega_{i+2}^*+\Omega_{2m-i}^*\Omega_{2m-(i+2)}=A_{i}A_{2+i}^*+A_{2m-i}^*A_{2m-(i+2)}.
		\end{equation}
		Similar computations for $N_{i,i+r}$, $r=3,4, \cdots, 2m-i$, yields that
		\begin{equation}\label{3} \Omega_{i}\Omega_{i+r}^*+\Omega_{2m-i}^*\Omega_{2m-(i+r)}=A_{i}A_{i+r}^*+A_{2m-i}^*A_{2m-(i+r)}.
		\end{equation}
		Hence by equations (\ref{normal Td}---\ref{3}), we conclude that $\bfA$ is normal if and only if for each $1\leq s,k\leq n-1$
		
		\begin{equation}\label{fin}
			\Omega_{s}\Omega^*_{k}+\Omega^*_{n-s}\Omega_{n-k}=A_s A_k^*+A_{n-s}^*A_{n-k}.
		\end{equation}
		A similar
		method works for the case $n =2m + 1$. Hence,
		we complete the proof.
	\end{proof}
	We end this section with the following example. 
	\begin{cor}\label{normal cor}
		If $\bfA=\bfT(A,\Omega)+\bfA_0\in\TT_{n}\otimes\AA$, such that $\bfA$ commute with $\bfS_X$  for some unitary  $X\in\AA^\prime$, then $\bfA$ is normal.
	\end{cor}
	\begin{proof}
		If $\bfA\bfS_X=\bfS_X\bfA$, then it follows from Theorem \ref{commutantSX}, that $\Omega=X^*\diamond\widetilde{A}$. Since $X\in\AA^\prime$ is unitary, then for every $1\leq s,k\leq n-1,$
		\begin{align*}\Omega_{s}\Omega^*_{k}+\Omega^*_{n-s}\Omega_{n-k}
			&=X^*A_{n-s}^*A_{n-k}X+A_{s}XX^*A_{k}^*\\
			&=A_{s}A_{k}^*+A_{n-s}^*A_{n-k}.
		\end{align*}
		Therefore , Theorem \ref{normal th100} implies that  $\bfA$ is normal.  
	\end{proof}
\begin{rmk}
	In the scalar case $\AA=\bbC$, Corollary \ref{normal cor} recaptures the normality of generalized  circulant matrices (see \cite{MAKDT}). 
\end{rmk}
	\section*{Acknowledgments}
	
	\end{document}